\documentclass[11pt]{amsart}
\usepackage{epsfig}
\usepackage{amscd}
\usepackage{amsmath}
\usepackage{amsxtra}
\usepackage{amsfonts}
\usepackage{amssymb}

\newtheorem{theorem}{Theorem}[section]
\newtheorem{corollary}[theorem]{Corollary}
\newtheorem{lemma}[theorem]{Lemma}
\newtheorem{proposition}[theorem]{Proposition}

\theoremstyle{definition}

\newtheorem{remark}[theorem]{Remark}

\newtheorem{example}[theorem]{Example}
\theoremstyle{remark}

\renewcommand{\theclaim}{\textup{\theclaim}}

\newtheorem*{acknowledgements}{Acknowledgements}

\numberwithin{equation}{section}

\def\openone

{\mathchoice

{\hbox{\upshape \small1\kern-3.3pt\normalsize1}}

{\hbox{\upshape \small1\kern-3.3pt\normalsize1}}

{\hbox{\upshape \tiny1\kern-2.3pt\SMALL1}}

{\hbox{\upshape \Tiny1\kern-2pt\tiny1}}}

\makeatletter

\newbox\ipbox

\newcommand{\ip}[2]{\left\langle #1\, , \,#2\right\rangle}
\newcommand{\diracb}[1]{\left\langle #1\mathrel{\mathchoice

{\setbox\ipbox=\hbox{$\displaystyle \left\langle\mathstrut
#1\right.$}

\vrule height\ht\ipbox width0.25pt depth\dp\ipbox}

{\setbox\ipbox=\hbox{$\textstyle \left\langle\mathstrut
#1\right.$}

\vrule height\ht\ipbox width0.25pt depth\dp\ipbox}

{\setbox\ipbox=\hbox{$\scriptstyle \left\langle\mathstrut
#1\right.$}

\vrule height\ht\ipbox width0.25pt depth\dp\ipbox}

{\setbox\ipbox=\hbox{$\scriptscriptstyle \left\langle\mathstrut
#1\right.$}

\vrule height\ht\ipbox width0.25pt depth\dp\ipbox}

}\right. }

\newcommand{\dirack}[1]{\left. \mathrel{\mathchoice

{\setbox\ipbox=\hbox{$\displaystyle \left.\mathstrut
#1\right\rangle$}

\vrule height\ht\ipbox width0.25pt depth\dp\ipbox}

{\setbox\ipbox=\hbox{$\textstyle \left.\mathstrut
#1\right\rangle$}

\vrule height\ht\ipbox width0.25pt depth\dp\ipbox}

{\setbox\ipbox=\hbox{$\scriptstyle \left.\mathstrut
#1\right\rangle$}

\vrule height\ht\ipbox width0.25pt depth\dp\ipbox}

{\setbox\ipbox=\hbox{$\scriptscriptstyle \left.\mathstrut
#1\right\rangle$}

\vrule height\ht\ipbox width0.25pt depth\dp\ipbox}

} #1\right\rangle}

\newcommand{\beq}{\begin{equation}}

\newcommand{\eeq}{\end{equation}}

\def\blfootnote{\xdef\@thefnmark{}\@footnotetext}

\input xy
\xyoption{all}
\usepackage{amssymb}

\newcommand{\abs}[1]{\lvert#1\rvert}

\newcommand{\norm}[1]{\lvert \lvert#1\rvert \lvert }

\def\R{\mathbb{R}}
\def\N{\mathbb{N}}

\def\-{^{-1}}

\def\C{\mathbb{C}}

\begin{document}

\title[Generalized Walsh bases and applications]{Generalized Walsh bases and applications}
\author{Dorin Ervin Dutkay}

\address{[Dorin Ervin Dutkay] University of Central Florida\\
	Department of Mathematics\\
	4000 Central Florida Blvd.\\
	P.O. Box 161364\\
	Orlando, FL 32816-1364\\
U.S.A.\\} \email{Dorin.Dutkay@ucf.edu}
\author{Gabriel Picioroaga}
\address{[Gabriel Picioroaga] University of South Dakota\\
          Department of Mathematical Sciences\\
          414 E. Clark St. \\
          Vermillion, SD 57069\\
U.S.A. \\} \email{Gabriel.Picioroaga@usd.edu}

\subjclass[2010]{42A38, 42C10,42C40,42-04}
\keywords{Conditional expectation, encoding, generalized Walsh functions, martingale, Maple, public key cryptography}

\begin{abstract} We investigate convergence properties of generalized Walsh series associated with signals $f\in L^1[0,1]$. We also show how the dependence of the generalized Walsh bases on $N\times N$ unitary matrices allows for applications in signal encoding and encryption, provided the signals are piece-wise constant on $N$-adic subintervals of $[0,1]$.

\end{abstract}
\maketitle \tableofcontents

\section{Introduction}
\par The Walsh basis functions form an orthonormal system that can be interpreted roughly as the discrete analog of classic sines and cosines. Unlike these the Walsh functions have several advantages: for example they take only two values $\pm 1$ on sub-intervals defined by dyadic fractions, thus making the computation of coefficients much easier. The Walsh functions  are connected to probability, e.g., the Walsh expansion can be seen as conditional expectation, and the partial sums form a Doob martingale. Moreover, partial sums converge a.e. for $L^1$ functions, which is not true of the classic exponential basis.
\par The  Walsh functions have found a wide range of applications: for example in modern communications systems (through the so-called Hadamard matrices, to recover information in the presence of noise and interference), signal processing (reconstruction of signals by means
of dyadic sampling theorems), and generally in computer science. For detailed accounts of the many areas of the applied sciences where the Walsh functions are used we refer the reader to the books \cite{AR}, \cite{Beach}, \cite{Har1}, and \cite{Har2}.
Next we describe briefly the classic Walsh system and some of its properties. These can be found in \cite{Walsh}, \cite{Fine} and references therein.  Let
$$\phi_n(t):=\left\{
\begin{array}{l l}
1, & \text{ if }t\in [0,\frac{1}{2^{n+1}})\cup  [\frac{3}{2^{n+1}},  \frac{4}{2^{n+1}})\cup\cdots\\
-1, & \text{ if }t\in [ \frac{1}{2^{n+1}}  ,\frac{2}{2^{n+1}})\cup  [\frac{4}{2^{n+1}},  \frac{5}{2^{n+1}})\cup\cdots\
\end{array}\right.$$
Usually the dyadic endpoints are not included as above and the value of $\phi_n$ is taken to be zero there (i.e. jump average). Same extension is considered for the Walsh functions. Here we are not affected by this as we are not concerned with convergence questions at such dyadic rationals. Now to define the $n$th Walsh function let $n=2^{n_1}+2^{n_2}+...+2^{n_k}$ be the base-2 expansion of $n$. Then the classic Walsh functions $(W_n)_{n\geq 0}$ defined by
$$W_n(t):=\phi_{n_1}(t)\phi_{n_2}(t)\dots \phi_{n_k}(t)$$
form an orthonormal basis (ONB) for the Hilbert space $L^2[0,1]$. There are certain features that make this ONB more desirable to work with than for example the Fourier system. As we will point out below in more detail, Walsh series associated to $f\in L^1[0,1]$ converge pointwise a.e. to $f$. This is also true for $f$ with bounded variation at a continuity point of $f$.\par
Various interpretations of the Walsh ONB have been given, in \cite{Fine}, \cite{Morgen}, and generalizations in \cite{Chr}, \cite{LiT}, \cite{Vil}. E.g. for the dyadic group $G$ the Walsh functions can be viewed as characters on $G$, or more generally starting with \cite{Vil}, as characters of a zero-dimensional, separable group. A generalized Walsh system based on $N$-adic numbers and exponentials functions can be found in \cite{Chr}, and  has been used to construct algorithms for polynomial lattices (a particular kind of digital net which in turn can be used in sampling methods for multivariate integration), see \cite{Dik1}, \cite{Dik2} and their references. However the Walsh-like system in \cite{DPS} which inspired the present work seems to be new: roughly, this new generalization of the Walsh ONB depends on certain unitary matrices (constant first row) and a simple IFS (iterated function system implemented by map $r$ below).
In \cite{DPS}, Theorem 3.1 gives a criteria to obtain ONBs based on Cuntz algebra representations. One byproduct (Proposition 3.10) recovers the classic Walsh ONB, and another (Theorem 3.11) generalizes it as follows: Start with an integer $N\geq 2$, and $A=[a_{ij}]_{i=0,N-1}^{j=0,N-1}$ a unitary matrix such that its first row entries are all equal to $\frac{1}{\sqrt{N}}$. For $0\leq i\leq N-1$ define
$$m_i(x):=\sqrt{N}\sum_{j=0}^{N-1}a_{ij}\chi_{[j/N, (j+1)/N)}(x)$$
Notice $m_0(x)=1$, $\forall x\in [0,1]$.
Denote by $r$ the map  $r(x):=(Nx)\mbox{mod}1=Nx-l \text{ if }x\in [l/N, (l+1)/N]$ where $l\in\{0,1,.., N-1\}$. With $n$ nonnegative integer written in its base-$N$ expansion $n=\sum_{k\geq 0}^{p-1}i_kN^k$,  the $n$'th Walsh function associated to matrix $A$ is
\beq\label{gw}
W_{n,A}(x):=m_{i_0}(x)m_{i_1}(rx)\dots m_{i_p}(r^{p-1}x)
\eeq
where $i_0, i_1,..,i_{p-1}$ are the non zero coefficients of the expansion;  $m_0\equiv 1$ so there's no need to display it in the product. Notice $W_{0,A}\equiv 1$.
When $N=2$ one obtains the classic Walsh system by picking the unitary matrix to be $$A:=\begin{pmatrix}
  1/\sqrt{2} &  1/\sqrt{2} \\
  1 & -1
 \end{pmatrix}$$
and the Rademacher functions are obtained as $\phi_n(t)=m_1(r^nt)$.
\par In the next section we prove convergence results for the generalized Walsh series formed with (\ref{gw}) (Theorem \ref{t1}, and Corollary \ref{uni}). In Example \ref{e1} we have implemented a generalized Walsh system with the aid of the mathematical software Maple to point out issues with the convergence of arbitrary generalized Walsh series. Corollary \ref{c1} helps define the discrete generalized Walsh transform and is very instrumental in our Maple computations. We end the section with Theorem \ref{mart}, and Corollary \ref{pconv} where we show that the connections with probability are still maintained: the generalized Walsh partial sums (of type $N^q$ ) form martingales, and their series behave as conditional expectations that converge in $L^p$.
\par While applications in signal processing could have been investigated, due to the multitude of generalized Walsh systems (each being associated to a unitary matrix) we found it natural to consider data encryption: in the last section of the paper we find a sufficient condition that two unitary matrices should satisfy in order to have secret communication in the spirit of public key cryptography, between two users each possessing a generalized Walsh transform. However as our remarks and examples indicate, a successful protocol depends on whether certain zero-dimensional systems of polynomial (quadratic) equations have infinitely many solutions.

\section{Pointwise Convergence}
We study convergence properties of the new orthonormal bases, and show that some of the convergence results from \cite{Kcz}, and \cite{Walsh} extend for any $N\geq 2$. For $N=2$  the theorem below was obtained by Walsh (with $f$ continuous), and Kaczmarz ($f\in L^1$), see also \cite{Fine} and references therein.

\begin{theorem}\label{t1} For $f\in L^1[0,1]$ the sequence of partial sums
$$S_{N^q}(x)=\sum_{n=0}^{N^q-1}\ip{f}{W_{n,A} }W_{n,A}(x)$$
converges a.e. to $f(x)$.
\end{theorem}
\begin{proof}
We show that if $x\in [0,1]$ is a Lebesque point of $f$ then the generalized Walsh series $S_{N^q}$ converges to $f(x)$. The calculations are based on the orthogonality conditions that the columns/rows of matrix $A$ satisfy. We will also write $W_n$ instead of $W_{n,A}$ (as long as we deal with a fixed $A$).
To not clutter our expressions we will consider the case $N=3$ and point out how/why the arbitrary dimensional analogue carries through.
We set out to prove a couple of properties of the Dirichlet kernel $D_q(x,t)$. Let us recall that in general $S_q=\int_0^1f(t)D_{q}(x,t)dt$ where
$D_{q}(x,t)=\sum_{n=0}^{q-1}W_n(x)\overline{W_n(t)}$. Notice first
\beq\label{di}
D_{3^q}(x,t)= \prod_{j=0}^{q-1}[1+m_1(r^jx)\overline{m_1(r^jt)}+ m_2(r^jx)\overline{m_2(r^jt)}]
\eeq
The relationship is easily checked by multiplying through all parentheses in the righthand side and using (\ref{gw}). If $N$ is arbitrary then the generic factor in the product above is of the form $[1+m_1(r^jx)\overline{m_1(r^jt)}+ m_2(r^jx)\overline{m_2(r^jt)}+...+m_{N-1}( r^jx )\overline{m_{N-1}(r^jt)}]$.
\par For $x\in [0,1]$ and $q\in\N$ there exists a unique $m=m(q, x)\in\{0,1,.., 3^q-1\}$ such that \\
$\alpha_{q,x}:=\frac{m}{3^q}\leq x< \beta_{q,x}:=\frac{m+1}{3^q}$.
We claim the following formula holds (with obvious $N^q$ replacement in the general case):

\beq\label{dixt}
D_{3^q}(x,t)=\left\{
  \begin{array}{l l}
    3^q & \quad \text{if } t\in (\alpha_{q,x},\beta_{q,x}) \\
   0 & \quad  \text{otherwise}
  \end{array} \right.
\eeq
Let $t\in (\alpha_{q,x},\beta_{q,x})$. With respect to base-3 expansion $m=k_{q-1}3^0+\dots+k_{q-1-j}3^j+\dots +k_03^{q-1}$, $k_j\in\{0,1,2\}$. Then for all $j\in\{0,1,.., 3^q-1\}$ we have $m_1(r^jx)=m_1(r^jt)=\sqrt{3}a_{1,k_j}$ and $m_2(r^jx)=m_2(r^jt)=\sqrt{3}a_{2,k_j}$. To see this notice that
$r^j(x)\in [k_{j}/3, (k_{j}+1)/3]$ for all $0\leq j\leq q-1$. (indeed, for $j=0$ apply the inequalities

$$\frac{k_{0}}{3}=\frac{k_03^{q-1}}{3^q}\leq \frac{m}{3^q} < \frac{1+m}{3^q}\leq \frac{  1+\sum_{j=0}^{q-2}2\cdot3^j +k_{0}3^{q-1} }{3^q }=\frac{1+k_0}{3}$$
then $x\in [k_0/3, (k_0+1)/3]$ and $\frac{m-k_03^{q-1}}{3^{q-1}} < r(x)=3x-k_0 < \frac{m-k_03^{q-1}+1}{3^{q-1}}$. One can continue with same lower- upper inequalities to get $r(x)\in [k_1/3, (k_1+1)/3]$, and so on).\\
From (\ref{di}) we obtain
$$D_{3^q}(x,t)=\prod_{j=0}^{q-1}[1+3|a_{1,k_j}|^2+ 3|a_{2,k_j}|^2]=\prod_{j=0}^{q-1}3[(\frac{1}{\sqrt{3}})^2+|a_{1,k_j}|^2+|a_{2,k_j}|^2]=\prod_{j=0}^{q-1}3=3^q$$
In the second product above the norm of the $k_j$'th column of matrix $A$ appears. Because $A$ is unitary this norm equals to $1$.

Assume now $t\notin [\alpha_{q,x},\beta_{q,x}]$. It follows that there exists a $j\in \{0,1,..., q-1\}$ such that $r^j(t)\notin [k_{j}/3, (k_{j}+1)/3]$. Then a factor in $D_{3^q}(x,t)$ from (\ref{di}) must be of the form
$1+3a_{1,k_j}\overline{a_{1,k_{j'}}} +3a_{2,k_j}\overline{a_{2,k_{j'}}}$ with $k_j\neq k_{j'}$. Rewriting it as
$3[ (1/\sqrt{3})^2+a_{1,k_j}\overline{a_{1,k_{j'}}} +a_{2,k_j}\overline{a_{2,k_{j'}}}]$ we recognize the inner product between the $k_j$'th and $k_{j'}$'th  columns of $A$. This last inner product of course vanishes as  $k_j\neq k_{j'}$ and $A$ is unitary. In conclusion (\ref{dixt}) holds. Next we estimate $\abs{f(x)-S_{3^q}(x)}$. For $x\in [0,1]$ we have

\beq\label{conv}
\abs{f(x)-S_{3^q}(x)}=\abs{ f(x)-\int_0^1f(t)D_{3^q}(x,t)dt }=\abs{ f(x)-3^q\int_{\alpha_{q,x}}^{\beta_{q,x} }f(t)dt }
\eeq
The last equality follows from (\ref{dixt}). Now when $q\rightarrow\infty$ the last term above converges to $0$ when $x$ is Lebesgue point for $f$, because $\beta_{q,x}-\alpha_{q,x}=1/3^q$. Thus for $f\in L^1[0,1]$ the convergence holds pointwise a.e.

\end{proof}

The corollary below was shown by Walsh for $N=2$. Because the Rademacher functions values are "jump-averaged" at points of discontinuity the $2^q$-type of Walsh sums converge at dyadic points. In our case this would mean that if $x=a$ is a $N$-adic rational then $S_{N^q}(a)=\sum_{n=0}^{N^q-1}\ip{f}{W_{n,A} }W_{n,A}(a)$ converges to $\frac{f(a+)+f(a-)}{2}$. We did not average out the discontinuities of the Rademacher-like functions $m_i(r^jx)$ above, and we will not emphasize  here the convergence of $S_{3^q}$ at $N$-adic rationals. Actually, what happens at a finite-jump discontinuity can be analyzed with the aid of the corollary below applied to $f_1$, and $f_2$ where

$
f_1(x)=\left\{
\begin{array}{l l}
f(x), & \text{ if } x<a\\
f(a-), & \text {if } x\geq a
\end{array}\right.
$
and
$
f_2(x)=\left\{
\begin{array}{l l}
f(x), & \text{ if }x>a,\\
f(a+), & \text{ if }x\leq a
\end{array}\right.
$.\\
Whether or not full generalized Walsh series (i.e. $lim_{k\rightarrow\infty}\sum_{n=0}^{k}\ip{f}{W_{n,A} }W_{n,A}$ ) converge to $f$ at a continuity point when $f$ is of  bounded variation seems to be quite a different matter from the classical case ( Theorem IV in \cite{Walsh}). Our Maple implementations of the generalized Walsh system show that the partial sums associated to a step function display a somehow erratic behavior ( Example \ref{e1} below and  Figure 1 and 2). We do not have yet an answer to this question and it might be possible that opposite to the classical Walsh system, the generalized Walsh partial sums associated to a bounded variation $f$ does not converge to $f$, even at continuity points.

\begin{corollary}\label{uni}If $f\in L^1[0,1]$  is continuous in a neighborhood of $x=a$ then the convergence in Theorem \ref{t1} is uniform inside an interval centered at $a$.

\end{corollary}

\begin{proof}
Let $[c,d]$ be an interval around $x=a$ on which $f$ is uniformly continuous such that $[c,d]=[k/N^{q_0},(k+3)/N^{q_0}]$, and $a\in [(k+1)/N^{q_0}, (k+2)/N^{q_0}]$. We show that $S_{N^q}$  converges uniformly to $f$ on the $N$-adic sub-interval $[(k+1)/N^{q_0}, (k+2)/N^{q_0}]$. For $\epsilon>0$ there exists $\delta_{\epsilon}>0$ such that $\abs{f(t)-f(t')}<\epsilon$ whenever $t,t'\in [c,d]$ and $\abs{t-t'}<\delta_{\epsilon}$. For $q_{\epsilon}\in\mathbb{N}$ large enough we have $1/N^q<\text{min}\{\delta_{\epsilon},1/N^{q_0}\}$ for all $q\geq q_{\epsilon}$. Then, with the notations in the proof of Theorem \ref{t1}, the interval $[\alpha_{q,t},\beta_{q,t}]$ is contained in the interval $[c,d]$ for any $t\in [(k+1)/N^{q_0}, (k+2)/N^{q_0}]$, and $q\geq q_\epsilon$. Using (\ref{conv}) we have
$$
\abs{S_{N^q}(t)-f(t)}=N^q\abs{\int_{\alpha_{q,t}}^{\beta_{q,t}}[f(\tau)-f(t)]d\tau}\leq\epsilon
$$
\end{proof}
Next corollary is easy to prove with the aid of  (\ref{dixt}) and (\ref{conv}), and can be used for data encoding/encrypting.
\begin{corollary}\label{c1} Assume $f:[0,1]\rightarrow\C$ is constant on the interval $I:=[i/N^q, (i+1)/N^q)$ for some $i\in\{0,1,..,N^q-1\}$, and $A$ is a unitary $N\times N$ matrix with constant $1/\sqrt{N}$ first row. Then for all $x\in I$ :
\beq\label{const}
f(x)=\sum_{n=0}^{N^q-1}\ip{f}{W_{n,A} }W_{n,A}(x)
\eeq
\end{corollary}

\begin{remark}Given positive integers $N\geq 2$ and $q$, and a matrix $A\in\mathcal{M}_{N\times N}(\mathbb{C})$ as above, one can define the discrete generalized Walsh transform $DTW_A:\mathbb{C}^{N^q}\rightarrow \mathbb{C}^{N^q}$ as follows
\beq\label{dwt}
DWT_A(v):=\left[\sum_{j=0}^{N^q-1}v_jW_{i,A}\left(\frac{2j+1}{2N^q}\right)\right]_{i=0}^{N^q-1}
\eeq
where $v=[v_j]_{j=0}^{N^q-1}$. Relation (\ref{dwt}) represents the sequence $[\ip{f}{W_{i,A}}]_{i=0,...,N^q-1}$ where $f$ is the function constant $v_i$ on each interval $[\frac{i}{N^q}, \frac{i+1}{N^q})$. The integration in each inner product translates into a finite sum because for $0\leq i \leq N^q-1$ the Walsh function $W_{i,A}$ is constant on intervals $(\frac{j}{N^q},\frac{j+1}{N^q})$, $\forall j=1,..., N^q-1$. Hence $\ip{f}{W_{i,A}}=\sum_{j=0}^{N^q-1}v_jW_{i,A}(t_j)$ where for $t_j$ we picked the midpoint of the interval $(\frac{j}{N^q},\frac{j+1}{N^q})$. Now with $x=\frac{k}{N^q}$, $0\leq k<N^q$ substituted in (\ref{const}) we obtain $v_k=\sum_{i=0}^{N^q-1}[DTW_A(v)]_iW_{i,A}\left(\frac{k}{N^q}\right)$ $\forall 0\leq k<N^q$,
i.e. $DWT_A$ is invertible.

\end{remark}

\begin{example}\label{e1}Consider the  step function  $$f(x)=\left\{
\begin{array}{l l}
0, & x\in [0,1/16)\cup [1/8, 3/16)\cup [1/4, 1/2)\\
1, & x\in [1/16,1/8)\cup [3/16, 1/4)\cup [1/2,1]
\end{array}\right.
$$
We have implemented the generalized Walsh system associated to the unitary matrix
$$A=
 \begin{pmatrix}
  \frac{1}{\sqrt{3}} &  \frac{1}{\sqrt{3}} &  \frac{1}{\sqrt{3}} \\
  \frac{\sqrt{2}}{2}& 0 &  -\frac{\sqrt{2}}{2}\\
 -\frac{\sqrt{6}}{6}  &  \frac{\sqrt{6}}{3} &  -\frac{\sqrt{6}}{6}
 \end{pmatrix}$$
By Theorem \ref{t1} the partial sums $\sum_{n=0}^{k}\ip{f}{W_{n,A} }W_{n,A}(x)$ converge to $f(x)$ for $k=N^q-1$. This is clearly the behaviour pictured in Figure 1 and 2 for $k=27$ and $k=81$ (notice that Corollary \ref{c1} is not applicable here and therefore the graphs for the partial sums having $3^q$ terms do not coincide with $f$'s, which is piecewise constant on a subdivision of $[0,1]$ coarser than one with triadic points). However even for a high number of terms  ( $k=300$ )  it seems  that the partial sums do not settle at $f(x)$.  Notice also that in our example $f$ has bounded variation.
\end{example}

\begin{figure}
 \includegraphics[width=1.5in]{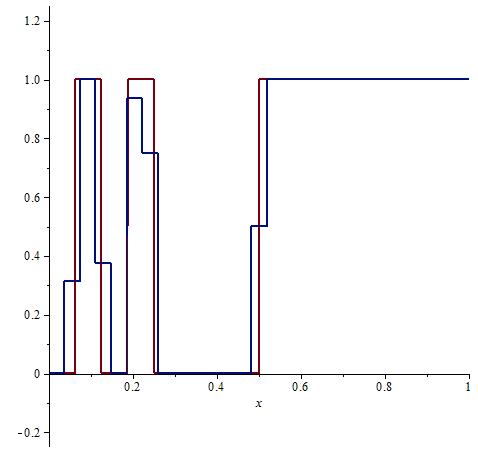}
   \includegraphics[width=1.5in]{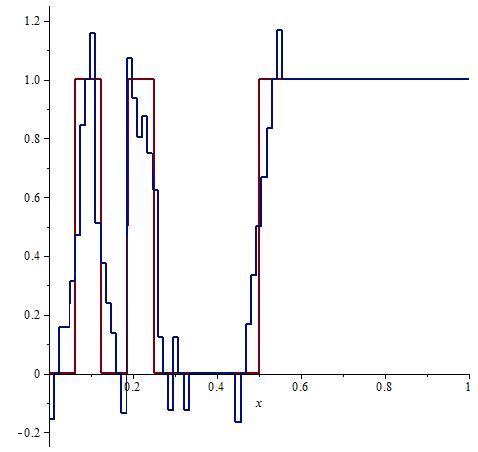}
    \includegraphics[width=1.5in]{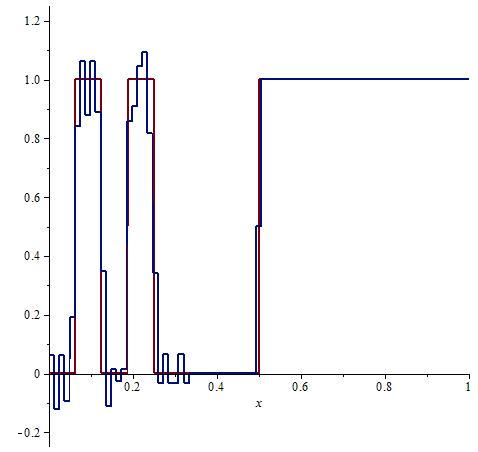}
    \includegraphics[width=1.5in]{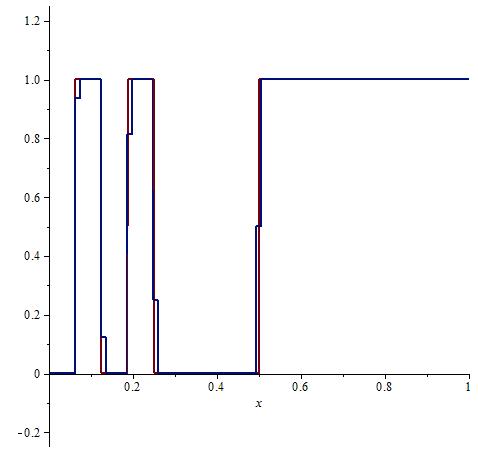}\\
\caption{ Graph of $f$ and its generalized Walsh sums with 27, 36, 60, and 81 terms}
\end{figure}

\begin{figure}
    \includegraphics[width=1.5in]{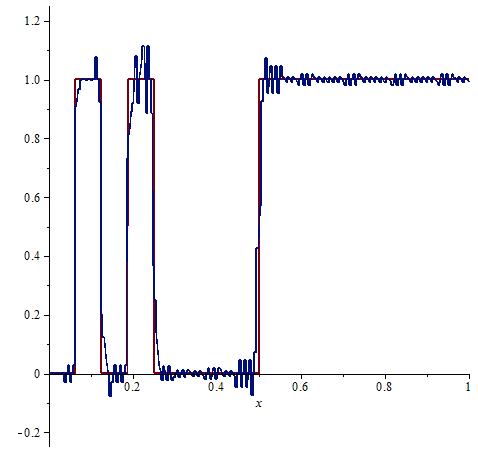}
    \includegraphics[width=1.5in]{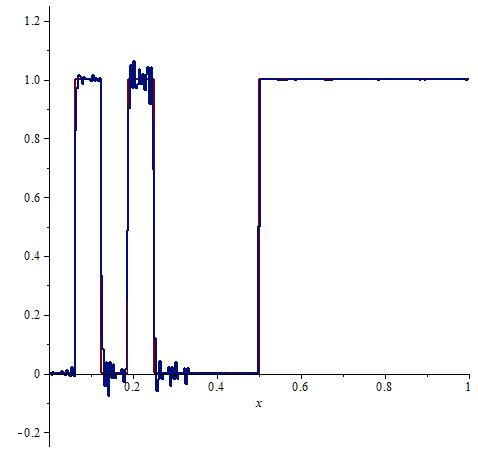}
\includegraphics[width=1.5in]{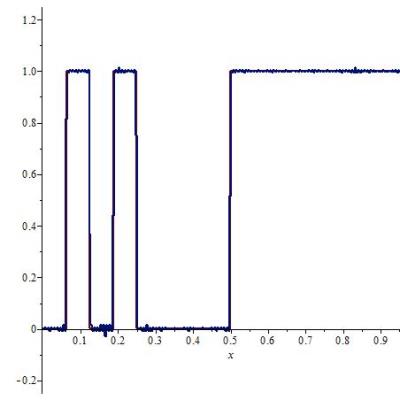}
 \includegraphics[width=1.5in]{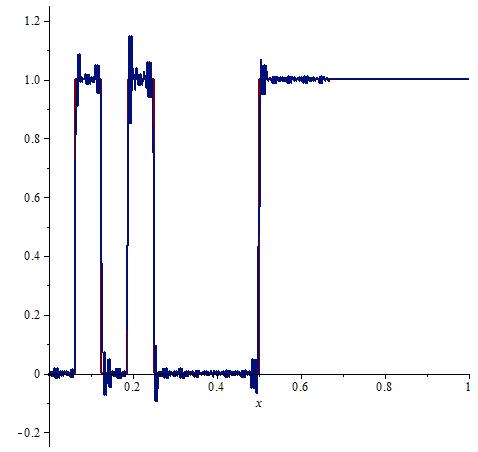}\\
\caption{Graph of $f$ and its generalized Walsh sums with 100, 200, 241, and 300 terms}
\end{figure}

\par Next we show that similarly to the classic Walsh functions the generalized Walsh expansions also can be interpreted as conditional expectations and Doob martingales. This enables us to conclude the convergence of the Walsh series in $L^p[0,1]$ with $1<p<\infty$.
\par Let $\mathcal{F}_q$ be the $\sigma$-algebra generated by the intervals $[j/N^q, (j+1)/N^q)$, $j=0,..., N^q-1$. Obviously $\mathcal{F}_1\subset\mathcal{F}_2\subset\dots\subset\mathcal{F}_q\subset\mathcal{F}_{q+1}\dots\subset\mathcal{B}$ where $\mathcal{B}$ is the Borel $\sigma$-algebra on $[0,1)$. Given a unitary matrix $A$ as before and a function $f$ we will denote by $S_{k}(f)$ the generalized Walsh series
$\sum_{n=0}^{k}\ip{f}{W_{n,A} }W_{n,A}$.

\begin{theorem}\label{mart} For $f\in L^1[0,1]$ the sequence $(S_{N^q}(f))_{q=1}^{\infty}$ is a martingale i.e.
\beq\label{exp}
\mathbb{E}[f|\mathcal{F}_q]=S_{N^q}(f)\text{, }\forall q\in\N
\eeq
\beq\label{mar}
\mathbb{E}[S_{N^{q+1}}(f)|\mathcal{F}_q]=S_{N^q}(f)\text{, }\forall q\in\N
\eeq
\end{theorem}

\begin{proof}

Let $x\in [0,1)$ and $m$ the unique number in $\{0,\dots,N^q-1\}$ such that $x\in [m/N^q, (m+1)/N^q)$. As in the proof of Theorem \ref{t1} we have $$S_{N^q}(f)(x)=N^q\int_{m/N^q}^{(m+1)/N^q}f(t)dt$$
Thus $S_{N^q}(f)$ is a piece-wise constant function, constant on each interval $[j/N^q, (j+1)/N^q)$, equal to the average of $f$ on that interval. Then one can see that, for $j=0,\dots, N^q-1$:
$$ \int_{0}^1f(t)\cdot\chi_{_{[\frac{j}{N^q}, \frac{j+1}{N^q})}}(t)=\int_0^1S_{N^q}(f)(t)\cdot\chi_{_{[\frac{j}{N^q}, \frac{j+1}{N^q})}}(t)dt$$
This proves (\ref{exp}). We get that $(S_{N^q}(f))_{q=1}^{\infty}$ is a martingale.

\end{proof}

\begin{lemma}\label{lem} The operator $f\rightarrow S_{N^q}(f)$ is bounded from $L^{p}[0,1]$ to $L^{p}[0,1]$ for all $1\leq p\leq\infty$.
\end{lemma}
\begin{proof} We have
\beq\label{unu}
\norm{S_{N^q}(f)}_1\leq \norm{f}_1
\eeq
\beq\label{doi}
\norm{S_{N^q}(f)}_{\infty}\leq \norm{f}_{\infty}
\eeq
Indeed, because $S_{N^q}(f)$ is piece-wise constant (the average of $f$) on $N$-adic intervals we have $$\int_0^1\abs{S_{N^q}(f)(t)}dt=\sum_{j=0}^{N^q-1}\abs{\int_{j/N^q}^{(j+1)/N^q}f(t)}dt\leq \int_0^1\abs{f(t)}dt$$
This proves (\ref{unu}). Also (\ref{doi}) follows from:
$$\abs{S_{N^q}(f)(t)}\leq N^q \int_{j/N^q}^{(j+1)/N^q}\abs{f(t)}dt\leq \norm{f}_{\infty}$$
The two inequalities imply that the operator $f\rightarrow S_{N^q}(f)$ is bounded between $L^{\infty}[0,1]\rightarrow L^{\infty}[0,1]$ and
$L^1[0,1]\rightarrow L^1[0,1]$. Then, by the Riesz-Thorin interpolation theorem, the operator $S_{N^q}$ is bounded from $L^{p}[0,1]$ to $L^{p}[0,1]$ for all $1\leq p\leq\infty$.

\end{proof}

\begin{corollary}\label{pconv} Let $1\leq p\leq \infty$, and $f\in L^p[0,1]$. Then $S_{N^q}(f)\rightarrow f$ a.e. in $[0,1]$. For $1<p<\infty$ we have  $S_{N^q}(f)\rightarrow f$ in $L^p[0,1]$.
\end{corollary}
\begin{proof}
Using Theorem \ref{mart} and Lemma \ref{lem} we have
$$\norm{\mathbb{E}[f|\mathcal{F}_q]}_p=\norm{S_{N^q}(f)}_p\leq\norm{f}_p\text{ for all }q\in\mathbb{N}\text{ and }p=1,\dots\infty$$
By Doob's martingale Convergence theorem we have $S_{N^q}(f)\rightarrow f$ a.e. in $[0,1]$ and $S_{N^q}(f)\rightarrow f$ in $L^p[0,1]$ for $1<p<1$.
\end{proof}

\section{An encryption protocol \'{a} la Diffie-Hellman}
\begin{remark}
Corollary \ref{c1} suggests the following encoding or encryption scheme: Given data recorded by a function $f$ which is piecewise constant on intervals of length $1/N^q$, compute the generalized Walsh coefficients $(\ip{f}{W_{n,A} })_{n=0}^{N^q-1}$, for a choice of unitary matrix $A$. One can generate unitary matrices $A\in\mathcal{M}_{3\times 3}(\R)$ with constant first row $1/\sqrt{3}$ by randomly choosing an entry $a\in [-\sqrt{2/3}, \sqrt{2/3}]$ in the second or third row and then solving for the remaining ones (we show how to implement such an algorithm using Maple software later in this section). The restriction $\abs{a}\leq \sqrt{2/3}$ comes from asking certain quadratic equations have solutions, and it can be easily observed by requiring the matrix be unitary.

One should hedge against brute force attacks to "guessing" the value $a$, which would act as secret key. For example if the range of $f$ is known (e.g. the alphabet letters are indexed from 1 to 26 and $f$ represents a message of length $3^q$) then one can estimate within a certain margin $|\tilde{a}-a|<\epsilon$ an approximate message
$$\tilde{f}(x)=\sum\ip{f}{W_{n,A}} W_{n, \tilde{A}}(x)$$
assuming that the finite sequence $(\ip{f}{W_{n,A}})_n$ representing $f$ has been intercepted, e.g. through an unsecure network.
Hence it may be safer to consider the process $f\rightarrow (\ip{f}{W_{n,A} })_{n=0}^{N^q-1}$  just an encoding step, which due to its complexity is suitable to further encryption (e.g. using classical  bit operations). For example one could add a perturbation $h(a, x)$ to the sequence encoding $f$, depending on the entry $a$ and other variables which may be part of the secret key.


\end{remark}
Of course security can be increased by allowing complex entries in $A$ (even though the data to be represented is made of real numbers). We note here that a scheme as above pertains to the area of symmetric key cryptography, i.e. both sender and receiver have knowledge of the matrix $A$ which generates the Walsh system. Generating such unitary $A$ is easily done with the aid of mathematical software, and the scheme described above can also be thought of as {\it{one time pad encryption}}.
\par In what follows we will study the theoretical feasibility of a protocol that shares similarities with both Diffie-Hellman key-exchange protocol and public key cryptography (RSA), based on generalized Walsh systems. More precisely we ask whether communication between Alice and Bob without sharing of the matrices $A$ and $B$ is possible. Our results indicate that some information about $A$ or $B$  has to be shared prior to message transmission (this theoretical "weakness" will be discussed later in the section). Hence this protocol is not "pure" Diffie-Hellman; the information to be shared (a system of quadratic polynomial equations) may be considered as public key, which the sender possesses (as opposed to common public key cryptography protocols where the receiver makes his public key known to anyone). For theoretical details regarding RSA, Diffie-Helmann key exchange protocols, and other public key cryptosystems and algorithms we refer the reader to \cite{GG} and \cite{DK}.
\begin{remark}
We describe first a more general set up: let $H_1$ be a non empty set (the space of {\it{messages}}) and $H_2$ another set (the space of {\it{encrypted messages}}). Then one can set up communication through an unsecure channel (e.g. Alice wants to send Bob a secret message $v\in H_1$ and Eve is capable to intercept all communications) without prior contact provided Alice and Bob are each in possession of {\it{operators}} $A:H_1\rightarrow H_2$ and $B:H_1\rightarrow H_2$ such that $B^{-1}\circ A\circ B^{-1}\circ A=I_{H_1}$, where $I_{H_1}$ is the identity operator (one might consider a slightly different   approach e.g. require that both $A$ and $B$ are defined on $H_1$ and adjust the above identity, and/or that their inverses exists and are defined on a smaller subspace). One should take care of a few requirements: for example the computation of $B^{-1}$ should be reliable (and easy); also there should be plenty of  operators $A$ and $B$  Alice and Bob could choose from without revealing their choice to each other or anyone else. Ideally there should be infinitely many $A$'s each of which admits infinitely many (or a large number of) $B$'s with $B^{-1}AB^{-1}A=I_{H_1}$. Such a family of operators is public and any pair (Alice, Bob) using the protocol will freely choose a pair $(A,B)$ with which they can start communicating. The situation where both choose the same operator pertains to the realm of symmetric key cryptography and consists only of the first two steps below (i.e. not all four); moreover such an occurrence would be improbable if there are infinitely many pairs $A$ and $B$ as above (see also remark \ref{theoretical} below). If all these are satisfied then one can start the exchange as follows: \\
1) Alice to Bob: $w_1=A(v)\in H_2$\\
2) Bob to Alice: $w_2=B^{-1}A(v)\in H_1$\\
3) Alice to Bob: $w_3=AB^{-1}A(v)\in H_2$\\
4) Bob applies $B^{-1}$ to $w_3$.\\
The scheme is safe provided Eve cannot decipher $v$ even when she is in possession of $w_i$, $i=1,2,3$. In other words Eve should not be able to figure out neither $A$ nor $B$ (or their inverses). Of course it is assumed that Eve is only {\it{eavesdropping}} without other interaction in the process (such as impersonating Alice and/or Bob). The question is where to look for such operators? One could encode the data to be transmitted in a finite dimensional  vector $v$, hence naturally the operators $A$ and $B$ may bee thought of as matrices. In this case we would have to find a infinite or very large number of matrices $C$ such that $C^2=I$, and decompositions of such matrices $C=AB$. Computationally we would have to deal with finding inverses of some of the matrices involved, which is unreliable when dealing with matrices of relatively large dimension. There are other ideas to consider, e.g. when the operators arise from the faithful representation of a group $G$ on a Hilbert space (such as $L^2(\mathbf{R}^n)$, which encodes the signals or the messages). Both encryption and security are then a consequence of the intrinsic properties of the group (generators and relations) and its representation (for example acting by 'shifting' should be considered unsecure).
\end{remark}
\begin{example}\label{e2}
Let $A$ and $B$ two unitary $N\times N$ matrices having constant $1/\sqrt{N}$ first row. We allow greater generality by working with complex numbers entries in both $A$ and $B$. We will denote by $\mathcal{W}_A$ the generalized Walsh transform associated to matrix $A$ i.e. $\mathcal{W}_A: L^2[0,1]\rightarrow l^2(\mathbb{N})$, $\mathcal{W}_A(f)=\ip{f}{W_{n,A}}_{n\geq 0}$, and similarly for the Walsh transform of $B$. Where well-defined (e.g. for finite sequences) the inverse transform operates as follows : if  $(a_n)_{n\geq 0}\in\l^2(\mathbb{N})$ then  $\mathcal{W}^{-1}_A((a_n)_n)=\sum_na_nW_{n,A} $. Also note that we have kept the same value $N$ in both systems, thus the same map $r(x)=(Nx)\mbox{mod}1$ enters in the definition of both Walsh ONBs. Then for a given $f\in L^2([0,1])$ the requirement
\beq\label{com}
\mathcal{W}^{-1}_B\circ\mathcal{W}_A\circ\mathcal{W}^{-1}_B\circ\mathcal{W}_A(f)=f
\eeq
has the following interpretation:
\vspace{0.1in}
\\
$\bullet $ Alice encrypts "message" $f$ using her matrix $A$ as the sequence $\mathcal{W}_A(f)=\ip{f}{W_{n,A}}_{n\geq 0}$ which she sends to Bob.
\vspace{0.1in}
\\
$\bullet$ Using his own matrix $B$, Bob constructs a new message $\mathcal{W}^{-1}_B\circ\mathcal{W}_A(f)=\sum_n \ip{f}{W_{n,A}} W_{n, B}(x)$ which he sends back to Alice. Notice that Bob sends a "whole" function whereas Alice sends a sequence, however in practice both $f$ and $\mathcal{W}^{-1}_B\circ\mathcal{W}_A(f)$ are piece-wise constant thus easy to record as finite length sequences.\vspace{0.1in}
\\
$\bullet$ Alice sends the coefficients $\mathcal{W}_A\circ\mathcal{W}^{-1}_B\circ\mathcal{W}_A(f)$ back to Bob.
\vspace{0.1in}
\\
$\bullet$ Bob finally applies the $\mathcal{W}^{-1}_B$  transform to the previous sequence and recovers the original $f$.
\vspace{0.1in}
\\
 Of course one should not expect that the intertwining relation (\ref{com}) just holds for any pair of matrices $A$ and $B$. We are interested in finding plenty of cases when it does. We would actually like to have infinitely many such pairs $(A,B)$ and to make it impossible to detect $A$ given $B$ or vice versa.  We will work under the assumption that for a fixed positive integer $q$, the "message" function $f$ is real valued, piecewise constant on $N$-adic intervals as in Corollary \ref{c1}. We have:
\begin{align*}
\mathcal{W}^{-1}_B\circ\mathcal{W}_A\circ\mathcal{W}^{-1}_B\circ\mathcal{W}_A(f)&=\sum_{k=0}^{N^q-1}\left( \sum_{l=0}^{N^q-1}\ip{f}{W_{l,A}}\ip{W_{l,B}}{W_{k,A}} \right)W_{k,B}\\
&=\sum_{l=0}^{N^q-1}\ip{f}{W_{l,A}}\sum_{k=0}^{N^q-1}\ip{W_{l,B}}{W_{k,A}}W_{k,B}
\end{align*}
Next we assume the "commutation" relation $\ip{W_{l,B}}{W_{k,A}}=\ip{W_{l,A}}{W_{k,B}}$. Hence we can continue the last equality with
$$\mathcal{W}^{-1}_B\circ\mathcal{W}_A\circ\mathcal{W}^{-1}_B\circ\mathcal{W}_A(f)=\sum_{l=0}^{N^q-1}\ip{f}{W_{l,A}}\sum_{k=0}^{N^q-1}\ip{W_{l,A}}{W_{k,B}}W_{k,B}$$
Notice that all generalized Walsh functions $x\rightarrow W_{l,A}(x)$ are piecewise constant on $N$-adic intervals so that Corollary \ref{c1} can be appplied:
$$\sum_{k=0}^{N^q-1}\ip{W_{l,A}}{W_{k,B}}W_{k,B}(x)=W_{l,A}(x)$$
We therefore obtain
$$\mathcal{W}^{-1}_B\circ\mathcal{W}_A\circ\mathcal{W}^{-1}_B\circ\mathcal{W}_A(f)=\sum_{l=0}^{N^q-1}\ip{f}{W_{l,A}}W_{l,A}=f$$
The last equality follows from Corollary \ref{c1} applied to $f$. We record these computations in the following
\end{example}
\begin{proposition}\label{p}
Let $q\in\mathbb{N}$, $N>1$ an integer. Then relation (\ref{com}) holds for any $f$ piecewise constant on each interval $[i/N^q, (i+1)/N^q]$, $i\in\{0,1,..., N^{q}-1\}$, provided
\beq\label{comm}
\ip{W_{l,B}}{W_{k,A}}=\ip{W_{l,A}}{W_{k,B}},\forall k,l=0,..N^q-1
\eeq
where $A$, and $B$ are unitary in $\mathcal{M}_{N\times N}(\mathbb{C})$, having constant $1/\sqrt{N}$ first row.
\end{proposition}
We would like to find a condition that is easier to implement in an algorithm than the above (\ref{comm}). Actually at this point it is not obvious that there should exist unitary matrices $A$ and $B$ satisfying (\ref{comm}). Note that the inner product in (\ref{comm}) is taken in the Hilbert space $L^2[0,1]$  whereas the one below in (\ref{coma}) is the usual $\mathbb{C}^N$ inner product.
\begin{theorem}\label{t2}Let $A=[a_{ij}]_{i=0,N-1}^{j=0,N-1}$, and $B=[b_{ij}]_{i=0,N-1}^{j=0,N-1}$ be unitary matrices in $\mathcal{M}_{N\times N}(\mathbb{C})$ with constant $1/\sqrt{N}$ first row. Using notation $\mbox{row }_{i,A}$ for the $i^{\mbox{th}}$ row in matrix $A$, condition (\ref{comm}) is equivalent to
\beq\label{coma}
\ip{\mbox{row }_{l,B}}{\mbox{row }_{ k,A}}=\ip{\mbox{row }_{ l,A}}{\mbox{row }_{ k,B}}\mbox{ for all }k, l\mbox{ in }\{1,2,3,...,N\}
\eeq

\end{theorem}
\begin{proof} The notation being already a bit crowded we will proceed under the assumption that $A$'s and $B$'s entries are real numbers. This will only affect not writing with conjugates when applying inner products (which will now be symmetric). The reader should easily retrace the proof and supply the complex conjugates where needed.
\par  The implication (\ref{comm})$\Rightarrow$(\ref{coma}) follows faster. First we will denote by $m_{i}^A$ the functions which define the Walsh system associated with unitary matrix $A$, and similarly for $B$.  Notice that the first rows of both $A$ and $B$ are written ${(a_{0,i})}_{i=0}^{N-1}$, and ${(b_{0,j})}_{j=0}^{N-1}$ (as in Introduction). Now when $0\leq k,l\leq N-1$ their base $N$ expansion are simply $k=k\cdot N^0$, and $l=l\cdot N^0$, hence $W_{i,C}(x)=m_i^C(x)$ for $i=k$, or $i=l$, and matrix $C=A$, or $C=B$. We have
$$\ip{W_{k,A}}{W_{l,B}}=\int_{0}^1m_k^A(x)m_l^B(x)=N\int_0^1\sum_{j=0}^{N-1}\sum_{i=0}^{N-1}a_{kj}b_{li}\chi_{I_j\cap I_i}(x)$$
where $I_j\cap I_i=[j/N, (j+1)/N]\cap [i/N, (i+1)/N]$. In conclusion
$$\ip{W_{k,A}}{W_{l,B}}=N\int_0^1\sum_{j=0}^{N-1}a_{kj}b_{lj}\chi_{I_j}(x)=\ip{\mbox{row }_{k+1,A}}{\mbox{row }_{ l+1,B}}$$
We thus obtain  (\ref{comm})$\Rightarrow$(\ref{coma}). We will show the converse with $N=3$, however the reader can easily replace the calculations for general $N$ as the pattern does not change much. When $k=0$ or $l=0$ or $k=l$ the relation (\ref{comm}) is true because of either orthogonality and $W_{0,A}=1=W_{0,B}$, or symmetry of the inner product over $\mathbb{R}^N$ . Also, when $k=1$ or $l=1$ (\ref{coma}) is true because of unitary requirements on $A$ and $B$. Hence the main assumption becomes
\beq\label{as}
\ip{\mbox{row }_{2,A}}{\mbox{row }_{ 3,B}}=\ip{\mbox{row }_{ 3,A}}{\mbox{row }_{ 2,B}}
\eeq
We want to deduce
 $$\int_0^1 m_{i_0}^A(x) m_{i_1}^A(rx)... m_{i_l}^A(r^lx) m_{j_0}^B(x) m_{j_1}^B(rx)... m_{j_p}^B(r^px)=$$
\beq\label{de}
= \int_0^1 m_{j_0}^A(x) m_{j_1}^A(rx)... m_{j_p}^A(r^px) m_{i_0}^B(x) m_{i_1}^B(rx)... m_{i_l}^B(r^lx)
\eeq
for all non negative integers $p$ and $l$, and for all $i_0, i_1,..., i_p$, $j_0, j_1,..., j_l$ in $\{0,1,2\}$. Following up on the discussion above, (\ref{de}) is true when ($l=0=p$) and ($i_0=0$ or $j_0=0$ or $i_0=j_0$). When $i_0\neq j_0$ and both non zero, these subscripts represent the $2^{nd}$ and $3^{rd}$ row of either $A$ or $B$, and (\ref{de}) follows from (\ref{as}). To better digest the proof of (\ref{de}) we will go through one more particular case, e.g. we will show
\beq\label{dep}
\int_0^1 m_{i_0}^A(x)m_{i_1}^A(rx)m_{j_0}^B(x)=\int_0^1 m_{j_0}^A(x)m_{i_0}^B(x)m_{i_1}^B(rx)
\eeq
Starting with the left-hand side we have
$$ \int_0^1 m_{i_0}^A(x)m_{i_1}^A(rx)m_{j_0}^B(x)=\sum_{k,l,t=0}^2a_{i_0,k}\chi_{[\frac{k}{3},\frac{k+1}{3}]}(x)\cdot a_{i_1,l}\chi_{[\frac{l}{3},\frac{l+1}{3}]}(rx)\cdot b_{j_0,t}\chi_{[\frac{t}{3},\frac{t+1}{3}]}(x)$$
Using notation $\lambda$ for the Lebesgue measure on $[0,1]$ we can continue with
\begin{align*}
 \int_0^1 m_{i_0}^A(x)m_{i_1}^A(rx)m_{j_0}^B(x)&=\sum_{k,l,t=0}^2a_{i_0,k}a_{i_1,l}b_{j_0,t}\cdot\lambda\left([\frac{k}{3},\frac{k+1}{3}]\cap r^{-1}([\frac{l}{3},\frac{l+1}{3}])\cap [\frac{t}{3},\frac{t+1}{3}]\right)\\
&=\sum_{k,l=0}^2a_{i_0,k}a_{i_1,l}b_{j_0,k}\cdot\lambda\left([\frac{k}{3},\frac{k+1}{3}]\cap r^{-1}([\frac{l}{3},\frac{l+1}{3}])\right)\\
&=\frac{1}{3^2}\sum_{k,l=0}^2a_{i_0,k}b_{j_0,k}a_{i_1,l}=\frac{1}{3^2}\sum_{k=0}^2a_{i_0,k}b_{j_0,k}\cdot\sum_{l=0}^2a_{i_1,l}
\end{align*}
In the calculations above we have used $\lambda\left([\frac{k}{3},\frac{k+1}{3}]\cap [\frac{l}{3},\frac{l+1}{3}]\right)=0$ whenever $k\neq l$, and \\
$\lambda\left([\frac{k}{3},\frac{k+1}{3}]\cap r^{-1}([\frac{l}{3},\frac{l+1}{3}])\right)=\frac{1}{9}$, for all $k,l\in\{0,1,2\}$ (this follows by inspecting the action of $r$ on $[0,1]$,  for example $rx\in [0,1/3]$ iff $x\in [0,1/9]\cup [1/3, 4/9]\cup [2/3, 7/9]$ etc). By applying  similar arguments  to the right-hand side of (\ref{dep}) we obtain
$$\int_0^1 m_{j_0}^A(x)m_{i_0}^B(x)m_{i_1}^B(rx)=\frac{1}{3^2}\sum_{k=0}^2a_{j_0,k}b_{i_0,k}\cdot\sum_{l=0}^2b_{i_1,l}$$
Because $A$, and $B$ are unitary with constant first row we have $\sum_{l=0}^2a_{i_1,l}=\sum_{l=0}^2b_{i_1,l}$ (= either 0 or $3/\sqrt{3}$, replaced by $N/\sqrt{N}$ in the general setting), and thus (\ref{dep}) follows due to (\ref{coma}) . Now to prove (\ref{de}) for any $l$ and $p$ we highlight the following property of $r$ which we mentioned in the particular case above. The reader can check it easily based on the observation that each set $r^{-1}[\frac{t}{3},\frac{t+1}{3}]$ contains precisely one component out of three of measure $1/9$ inside any interval $[\frac{k}{3},\frac{k+1}{3}]$, where $k,t\in \{0,1,2\}$. Hence, if $l\leq p$ are nonnnegative integers, and $t_0,t_1,...,t_l$,$q_0,q_1,...,q_p$ are $\{0,1,2\}$-digits, the Lebesgue measure of the set
\begin{align*}
\mathcal{S}_{t_0,t_1,...,t_l,q_0,q_1,...,q_l,...,q_p}:&=\left[\frac{t_0}{3},\frac{t_0+1}{3}\right]\cap r^{-1}\left[\frac{t_1}{3},\frac{t_1+1}{3}\right]\cap...\cap r^{-l+1}\left[\frac{t_l}{3},\frac{t_l+1}{3}\right]\cap \\
& \cap \left[\frac{q_0}{3},\frac{q_0+1}{3}\right]\cap r^{-1}\left[\frac{q_1}{3},\frac{q_1+1}{3}\right]\cap...\cap r^{-p+1}\left[\frac{q_p}{3},\frac{q_p+1}{3}\right]\\
\end{align*}
is obtained
\beq\label{mes}
\lambda(\mathcal{S}_{t_0,t_1,...,t_l,q_0,q_1,...,q_l,...,q_p})=\left\{
\begin{array}{l l}
\frac{1}{3^{p+1}} & \mbox{ if }t_0=q_0\mbox{ and }t_1=q_1\mbox{ and...and }t_l=q_l\\
\\
0 & \mbox{ otherwise}
\end{array}
\right.
\eeq
Without loss of generality assume $l\leq p$ and start with the left-hand side (LHS) of (\ref{de}). Replacing the $m$'s and integrating the characteristic functions we obtain
$$\mbox{LHS}=\sum_{\substack{t_0,t_1,...,t_l=0\\ q_0,q_1,...,q_p=0}}^2a_{i_0,t_0}a_{i_1,t_1}\dots a_{i_l,t_l}\cdot b_{j_0,q_0}b_{j_1,q_1}\dots b_{j_p,q_p}\cdot\lambda(\mathcal{S}_{t_0,t_1,...,t_l,q_0,q_1,...,q_l,...,q_p})$$
Using (\ref{mes}) we continue with
\begin{align*}
\mbox{LHS}&=\frac{1}{3^{p+1}}\sum_{t_0,t_1,...,t_l=0}^2a_{i_0,t_0}b_{j_0,t_0}\dots a_{i_l,t_l}b_{j_l,t_l}\cdot\sum_{q_{l+1},..., q_p=0}^2b_{j_{l+1},q_{l+1}}\dots b_{j_p,q_p} \\
&=\frac{1}{3^{p+1}}\prod_{x=0}^{l}\ip{row_{i_x,A}}{row_{j_x,B}}\cdot \prod_{y=l+1}^p(b_{j_y,0}+b_{j_y,1}+b_{j_y,2})
\end{align*}
Now due to (\ref{coma}) we may switch $A$ with $B$ in the first product above. As for the second product, notice the sum of the ${j_y}^{th}$ row of matrix $B$: each such sum is equal to the sum of the ${j_y}^{th}$ row of matrix $A$ (both being equal to either 0 or $3/\sqrt{3}$), according to perpendicularity requirements. Therefore
\begin{align*}
\mbox{LHS}&=\frac{1}{3^{p+1}}\prod_{x=0}^{l}\ip{row_{i_x,B}}{row_{j_x,A}}\cdot \prod_{y=l+1}^p(a_{j_y,0}+a_{j_y,1}+a_{j_y,2})\\
&=\frac{1}{3^{p+1}}\sum_{t_0,t_1,...,t_l=0}^2b_{i_0,t_0}a_{j_0,t_0}\dots b_{i_l,t_l}a_{j_l,t_l}\cdot\sum_{q_{l+1},..., q_p=0}^2a_{j_{l+1},q_{l+1}}\dots a_{j_p,q_p}
\end{align*}
The last term we have obtained is equal to the right-hand side RHS of (\ref{de}), as we can express it using (\ref{mes}) precisely in the same way we started with LHS. In conclusion (\ref{de}) follows from (\ref{coma}) and we are done.
\end{proof}

\begin{remark}\label{theoretical} We describe next the theoretical framework underlying a possible cryptographic protocol based on the ideas above.
\begin{enumerate}
\item To obtain a generalized Walsh ONB, Alice sets up the following equations:

\begin{align*}\label{aliceq}
  &1)\quad a_{1,j}=1/\sqrt{N}, \quad\forall j=1,...,N\\
&2)\quad \sum_{j=1}^N|a_{i,j}|^2=1, \quad\forall i=2,...,N\\
 &3)\quad \sum_{j=1}^Na_{i,j}=0,  \quad\forall i=2,...,N\\
  &4\quad) \sum_{k=1}^Na_{i,k}\overline{a_{j,k}}=0, \quad\forall 1<i<j\leq N
  \end{align*}

Discarding the first item we are left with $\frac{N^2-N}{2}+N-1$ equations with $N(N-1)$ unknowns $a_{i,j}$, $i=2,...,N, j=1,...,N$.
If $a_{i,j}\in\mathbb{R}$ then one obtains a system of $\frac{N^2-N}{2}+N-1$ polynomial (quadratic) equations with infinitely many solutions as long as $N\geq 3$. All Alice has to do is pick a few prescribed entries (with some care so as to maintain norm $1$ on the row the entry comes from) and solve for the remaining entries (see example below). Of course one can allow for complex unknowns $a_{ij}$ with non zero imaginary parts. In this case the system above can again be thought of as a system of polynomial equations with real value unknowns by splitting each equation into real and imaginary parts. Actually in this case the number of unknowns doubles (each $a_{i,j}$ contributes two more unknowns, real and imaginary) while the equations in item ii) above do not. Of course it was obvious that there are infinitely many unitary matrices but what we spelled out here was the precise requirements we need in order to implement in a computer.
 \item Next comes the "sharing" part: obviously Alice must not reveal $A$ but she will have to "help" Bob choose the right matrix $B$, i.e. such that (\ref{coma}) holds. We will assume all entries are real numbers although one can adjust to complex ones as well. Hence in relationship (\ref{coma}) the symmetries can be discarded and only $\frac{N^2-N}{2}$ equations will be relevant: it means that Alice sends Bob the following system of equations in unknowns $b_{ij}$, $i=2,...,N, j=1,...,N$:

$$
5)\quad\sum_{j=1}^{N}a_{kj}b_{lj}=\sum_{j=1}^{N}a_{lj}b_{kj}, \quad\forall 1<l<k\leq N
$$
In the above equations Alice's secret key, $a_{ij}$ seems to be revealed: of course this would be very damaging but Alice can simply multiply each equation by a random number thus masking her matrix.\\
\item Bob considers a system of equations similar to Alice's above to which he adds item 5). When working with real values Bob deals with a system of $N^2-N$ polynomial equations and $N^2-N$ unknowns ($b_{ij}$, $i=2,...,N, j=1,...,N$). This system of equations clearly has solutions (i.e. Alice's own $a_{ij}$) but it is important to get a large number of (possibly infinitely many) solutions. This will hedge against an eavesdropper detecting matrix $B$. In an example below we display such a situation using Maple (infinitely many matrices $B$ corresponding to a given $A$); however we would like to obtain infinitely many $A$ for which there are infinitely many $B$ satisfying all items 1) through 5). It is not our scope here to go into a thorough study of the system of equations above, nevertheless we feel it is a very interesting question to settle the existence of infinitely many examples of matrices $A$ and $B$ as above. E.g. Maple is capable to calculate Gr\"{o}bner bases (theoretical tool that among other things tells whether a zero-dimensional system of polynomial equations has finitely many solutions) and find approximate solutions for the systems above. Finding Gr\"{o}bner bases is based on Buchberger algorithm and the process is time consuming even for $N=4$ ( see \cite{Bu}, also the Help section in Maple which contains practical details on these bases and more efficient algorithms).
\item Suppose $f=(a_1, a_2,\dots, a_{N^q})$ represents the secret message to be transmitted by Alice to Bob.
From the previous steps we obtained $A$ and $B$ that satisfy (\ref{coma}). By Theorem \ref{t2} equation (\ref{com}) holds. Now the communication continues as in Example \ref{e2}.
\end{enumerate}
\end{remark}

\begin{example}\label{e3} In this example the matrix $A$ allows for infinitely many matrices $B$ for which the above relation (\ref{coma}) holds. We have experimented with other $3\times 3$ and $4\times 4$ matrices $A$ for which Maple did not find infinitely many $B$ satisfying equation $5)$ in Remark \ref{theoretical}. At this point we do not know if such occurrences are rare, and we do not have yet a theoretical tool to characterize all such matrices.

\begin{itemize}
\item Alice has matrix $$A=
 \begin{pmatrix}
  \frac{1}{\sqrt{3}} &  \frac{1}{\sqrt{3}} &  \frac{1}{\sqrt{3}} \\
  \frac{\sqrt{2}}{2}& 0 &  -\frac{\sqrt{2}}{2}\\
 -\frac{\sqrt{6}}{6}  &  \frac{\sqrt{6}}{3} &  -\frac{\sqrt{6}}{6}
 \end{pmatrix}$$
Equation (\ref{coma}) in this case becomes $p2^{-1/2}-r2^{-1/2}+x6^{-1/2}-y(1/3)6^{1/2}+z6^{-1/2}=0$, and is sent to Bob, after multiplication by a random number (to mask it from possible eavesdroppers).
\item Replacing Bob's unknowns $[b_{i,j}]_{i=2,3}^{j=1,2,3}$ by $x,y,z,p,q,r$ the following system must be solved:
$$ x^2+y^2+z^2-1=0,\quad p^2+q^2+r^2-1=0$$
$$ x+y+z=0,\quad p+q+r=0$$
$$xp+yq+zr=0,\quad p2^{-1/2}-r2^{-1/2}+x6^{-1/2}-y(1/3)6^{1/2}+z6^{-1/2}=0$$
In this case (i.e. for $A$ above) Maple solve command finds infinitely many solutions indexed by (the free) variable ${\bf{r}}$ below (indeterminate $Z$ is a place-holder for the unknown in the quadratic equations):
$$ p = r+(1/2)\sqrt{2}RootOf(6Z^2+6r^2+3\sqrt{2}Z\sqrt{6}r-1)\sqrt{6}$$ $$q = -2r-(1/2)\sqrt{2}RootOf(6Z^2+6r^2+3\sqrt{2}Z\sqrt{6}r-1)\sqrt{6}$$
 $${\bf{r = r}}, x = RootOf(6Z^2+6r^2+3\sqrt{2}Z\sqrt{6}r-1)+(1/2)\sqrt{2)}\sqrt{6}r
$$
$$y = RootOf(6Z^2+6r^2+3\sqrt{2}Z\sqrt{6}r-1)$$
$$z = -(1/2)\sqrt{2}\sqrt{6}r-2RootOf(6Z^2+6r^2+3\sqrt{2}Z\sqrt{6}r-1)$$
\item Bob picks a value $r$ such that the quadratic equation $6Z^2+6r^2+3\sqrt{2}Z\sqrt{6}r-1=0$ has real solutions in indeterminate $Z$. Notice that such values can be chosen randomly in a subinterval of $[-1,1]$. E.g. for $r=0.2$  Bob sets up a matrix
    $$B=\begin{pmatrix}
    3^{-1/2} & 3^{-1/2} & 3^{-1/2}\\
    -.2226063221 & -.5690164837 & .7916228058\\
    -.7855654600 & .5855654600 & .2    \end{pmatrix}$$
Maple finds a numeric approximation when solving the systems of polynomial equations (it considers them with rational coefficients). Thus matrix $B$ is "almost" unitary. E.g. Maple gives the following computation :
$$B^*B= \left( \begin {array}{ccc}  0.9999999998& 0.0000000001& 0.0
\\ \noalign{\medskip} 0.0000000001& 0.9999999999&- 0.0000000001
\\ \noalign{\medskip} 0.0&- 0.0000000001& 1.0\end {array} \right)
$$
\item The signals/messages to be transmitted must be of length $3^q$. For example let $f=000 11 00000 1111111 000 2222222$ be a signal of length $27$ which is encoded as a step function

$$
f(x)=\left\{
\begin{array}{l l }
0, & \text{if } 0\leq x < 1/9 \\
1, & \text{if } 1/9 \leq x < 5/27\\
0, & \text{if } 5/27 \leq x < 10/27\\
1, & \text{if } 10/27 \leq x < 17/27\\
0, & \text{if } 7/27 \leq x < 20/27\\
2, & \text{if } 20/27\leq x <1
\end{array}\right.$$

\begin{center}
\includegraphics[width=2in]{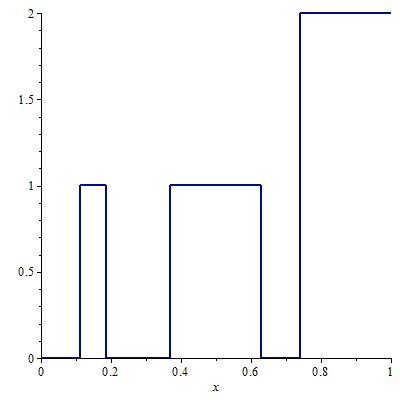}
\end{center}
i) The sequence $W_Af$ (without its first Walsh coefficient) which Alice sends to Bob:
\begin{align*}
&-0.5443310539,-0.05237828008,-0.1814436847, 0.2222222222, 0.1283000598,\\
&0.2618914004,0., -0.07407407407, -0.04536092117, 0.1666666667, 0.03207501497,\\
&-0.2222222222, 0.1360827635, -0.07856742012, 0.1283000598, 0., -0.09072184234,\\
&0.02618914004, 0.09622504490, 0.09259259259, -0.06415002993, 0.07856742012,\\
&0.04536092117, 0.03703703704, 0., -0.1047565601
\end{align*}

ii) Maple display of the graph of $W_B^{-1}W_Af$, which is sent to Alice:

\begin{center}
\includegraphics[width=2in]{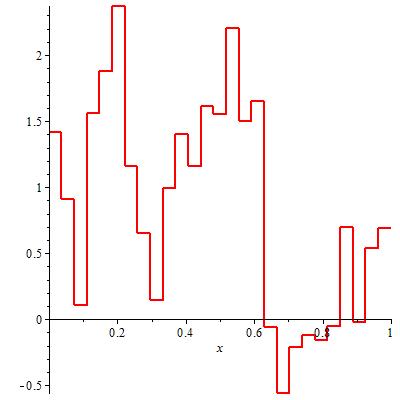}
\end{center}

iii) Alice applies her Walsh transform to Bob's function and sends him the sequence $W_AW_B^{-1}W_Af$:
\begin{align*}
&0.4268793977, 0.3417802807, -0.05238646443, 0.1424437841, 0.08209867227,\\
&0.3142683164, 0.2103987320, 0.005704364048, 0.01428020223, 0.1948148148,\\
&0.06725825079, -0.06424603320, 0.001060076263, -0.1578695927, -0.2267207154\\
&-0.1331355569, -0.01534027859, 0.05039404776, 0.003108220861, 0.06444444442\\
&-0.03427062570, 0.01804658637, -0.05818088527, -0.1209391520, 0.02910416584\\
&-0.06844063412
\end{align*}

iv) Maple graphs of $W_B^{-1}W_AW_B^{-1}W_Af$ and $f$ coincide . This is recovered by Bob by applying $W_B^{-1}$ to the sequence in iii):
\begin{center}
\includegraphics[width=2in]{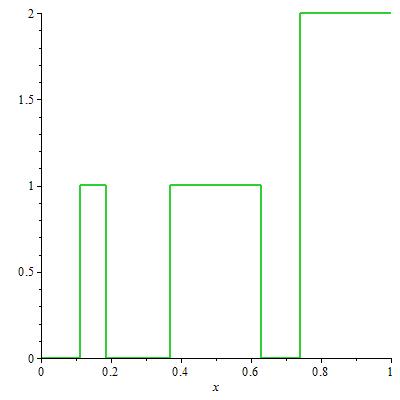}
\end{center}

\end{itemize}
One should add that Maple displays both graphs as "equal" which from the point of view of reading off values in the range $\{0,1,2\}$ is quite sufficient. Computationally the functions are {\it{almost}} equal. For example, in our Maple program we evaluated the value $$W_B^{-1}W_AW_B^{-1}W_Af(0.4)=1.185185185-.1069167165\sqrt{3}\sim 0.9999999998\sim 1=f(0.4)$$
\end{example}

\begin{acknowledgements}
The second named author would like to thank Professors Jose Flores for many discussions about Maple, and Catalin Georgescu for helpful insights related to Gr\"{o}bner bases.
\end{acknowledgements}

\bibliographystyle{alpha}	
\bibliography{eframes}

\end{document}